\pdfoutput=1
\documentclass[a4paper,10pt]{amsart}

\usepackage{amsmath,amssymb,amsthm}
\usepackage{tikz}
\usepackage{geometry}
\usepackage[vcentermath]{youngtab}
\usetikzlibrary{positioning}

\newcommand{\dd}{\delta}
\newcommand{\N}{\mathbb{N}}
\newcommand{\Z}{\mathbb{Z}}
\newcommand{\Q}{\mathbb{Q}}
\newcommand{\J}{\mathbf{J}}
\newcommand{\T}{\mathcal{T}}
\newcommand{\Sp}{\mathrm{Sp}}
\newcommand{\fkt}{\mathfrak{t}}
\newcommand{\fks}{\mathfrak{s}}
\newcommand{\fku}{\mathfrak{u}}
\newcommand{\sym}[1]{\mathfrak{S}_{#1}}
\newcommand{\uu}[1]{\underline{#1}}
\newcommand{\ses}[3]{0\rightarrow #1\rightarrow #2 \rightarrow #3\rightarrow 0}
\newcommand{\op}{\mathrm{op}}
\newcommand{\tp}{\mathrm{tp}}

\renewcommand{\leq}{\leqslant}
\renewcommand{\geq}{\geqslant}
\renewcommand{\phi}{\varphi}

\renewcommand{\Im}{\mathrm{Im}}

\makeatletter
\newcommand{\pushright}[1]{\ifmeasuring@#1\else\omit\hfill$\displaystyle#1$\fi\ignorespaces}
\newcommand{\pushleft}[1]{\ifmeasuring@#1\else\omit$\displaystyle#1$\hfill\fi\ignorespaces}
\makeatother

\theoremstyle{plain}
\newtheorem{thm}{Theorem}
\newtheorem{prop}[thm]{Proposition}

\newtheorem{lem}[thm]{Lemma}

\newtheorem*{thm*}{Theorem}
\theoremstyle{definition}
\newtheorem{defn}[thm]{Definition}
\newtheorem{rem}[thm]{Remark}
\newtheorem{eg}[thm]{Example}
\newtheorem*{rem*}{Remark}

\title{On central idempotents in the Brauer algebra}
\author{O. H. King \and P. P. Martin \and A. E. Parker}
\address{Department of Mathematics \\ University of Leeds \\ Leeds, LS2 9JT \\ UK}
\email{O.H.King@leeds.ac.uk}
\email{ppmartin@maths.leeds.ac.uk}
\email{A.E.Parker@leeds.ac.uk}

\begin{document}
\begin{abstract}
  We provide a method for constructing central idempotents in the Brauer algebra relating to the splitting of certain short exact sequences. We also determine some of the primitive central idempotents, and relate properties of the idempotents to known facts about the representation theory of the algebra.
\end{abstract}
\maketitle

\section{Introduction}
\label{sec:intro}

The Brauer algebra $B_n(\delta)$ was introduced by Brauer in
\cite{brauer}.
Brauer used it in place of the symmetric group
$\sym{n}$ when
constructing an analogue of Schur-Weyl duality for the orthogonal and
symplectic groups. 
In this construction the parameter $\delta$ is chosen to match the
dual group.
Since then the algebra has been 
studied as a $k$-algebra in its own right, with $\delta\in k$. 
In particular, 
it can be defined as a  $\Z[\dd]$-algebra.
Brown \cite{brown} showed that the algebra is generically semisimple
over the complex field and 
effectively constructed modules for the  $\Z[\dd]$-algebra that pass
to the simple modules in this case.
This allows us to study the algebra using a modular system \cite{Brauer39}.
Building on work of Wenzl \cite{wenzl}, Rui \cite{rui} gave an
explicit semisimplicity criterion for each $n$. 
There has been much subsequent work in determining the structure of 
lifts of Brown's ordinary irreducible modules
in the non-semisimple cases
(see for example \cite{HartmannPaget06,cdm} for references).
Graham and Lehrer
\cite{grahamlehrer} showed that 
Brown's modules are the cell modules of 
$B_n(\dd)$   
regarded as a cellular algebra. 
Finally, the Cartan decomposition matrix was computed in the complex
case \cite{martin}. 
This triggered interesting works on geometric aspects such as 
\cite{CoxDevisscher,Stroppel}.
However the case of characteristic $p > 0$ remains
essentially open (cf. \cite{Shalile,DevisscherMartin}), 
and it is this problem that motivates the present work.

In characteristic zero, the representation theory of the Brauer
algebra is already quite rich (see for instance \cite{cdm},
\cite{martin}). 
In the non-semisimple cases we can find cell modules of arbitrary Loewy
length by increasing $n$. In positive characteristic it becomes
dramatically more complicated still, due in part to the fact that it
contains the representation theory of the symmetric group. 

As noted, the Brauer algebra 
may be  
defined as a $\Z[\dd]$-algebra,
i.e. over a commutative ring with a single parameter. 
This allows us to define integral forms of the cell
modules, which allows for independent specialisation of the parameter
and the field via extension of scalars. In our case, we will need to
consider a ring where the above is possible, but where we can also
invert certain monic polynomials in $\dd$. This ring $K$ will be
introduced in Section \ref{sec:brauer}. 

A useful realisation of the Brauer algebra is that as a diagram
algebra, i.e. as a subalgebra of the partition algebra. 
In this way, a basis of the algebra is the set of graphs with $2n$
nodes such that each node is joined to precisely one other by a single
edge. This allows us to define ideals based on a so-called propagating
number, and also view the group algebra of the symmetric group
$\sym{n}$ inside $B_n(\dd)$. These concepts will be made precise in
Section \ref{sec:brauer}. 

In the general setting of algebraic representation theory, 
a fundamental question asks how the
representation theory of   an algebra 
with a splitting modular system \cite[\S1.9]{benson}
is manifested in the primitive
central idempotents
of its ordinary case. 
For instance in a unital algebra $\Lambda$ over
an Artinian ring, a decomposition of $1_\Lambda$ into a sum of
primitive central idempotents gives the blocks of the algebra
(see e.g. \cite[\S1.8]{benson}). 
In this paper we will construct a family
of central idempotents of $B_n(\dd)$. 
We then  
use existing information
about the Brauer algebra to establish connections to representation
theory.

In principle there are several possible 
approaches to finding idempotents in the Brauer algebra. 
A result of Kilmoyer \cite[Proposition 9.17]{curtisreiner} allows one
to use the characters of a semisimple algebra to construct its
primitive central idempotents. Leduc and Ram \cite{leducram} express
the Brauer algebra as a multimatrix algebra and then give a method for
finding primitive central idempotents by considering pairs of paths in
the Bratteli diagram. 
Isaev and Molev \cite{isaevmolev} build on this by introducing a
method based on the `Jucys-Murphy elements' of the Brauer algebra. A recent paper of Doty, Lauve and Seelinger \cite{Doty} determines the central idempotents in so-called multiplicity free families of algebras, of which the Brauer algebra is one.
However for the types of questions we would like to ask regarding the Brauer algebra, the method we describe in this paper has several
advantages over the existing ones. In particular our choice of the
ring $K$ gives us control over the form of the coefficients appearing
in each idempotent. This allows us to easily draw conclusions about
the representation theory of the algebra before we have arrived at the
final result. 
Use of Kilmoyer's proposition requires us to know the characters of the Brauer algebra, which in itself is a non-trivial task, and the method employed by Leduc and Ram becomes rather inefficient as $n$ increases, and moreover is only valid over $\overline{\Q(\delta)}$. As such it is difficult to see any results regarding the integral or exceptional representation theory of the algebra until the process has finished. Similar comments can be made about Isaev and Molev's method, where interim steps are also based upon paths in the Bratteli diagram. Use of the Bratteli diagram, which becomes increasingly complicated as $n$ grows, is also necessary for Doty, Lauve and Seelinger's method.

Our approach mirrors that of \cite{martinwoodcock}, in that we
construct splitting idempotents of certain exact sequences. 
We will
see later that a short exact sequence of $\Lambda$-bimodules
\begin{equation}\ses{J}{\Lambda}{\Lambda/J}\label{eq:ses}\end{equation}
splits if and only if there is an element $\phi_J\in \Lambda$ satisfying
\begin{enumerate}
\item $\phi_J\equiv1_\Lambda$ mod $J$, and
\item $J\phi_J=\phi_JJ=0$.
\end{enumerate}
If $\phi_J$ exists then it is unique and is a central idempotent in
$\Lambda$. We call it the splitting idempotent of the sequence
\eqref{eq:ses}. In our case, we will consider ideals
$\overline{J_n}(\ell)$ of $B_n(\dd)$ generated by diagrams with $\ell$
or fewer propagating lines (see Section \ref{sec:brauer}).
We use as a labelling set, tableaux with entries from the set $\{N,S,
P\}$ under an equivalence, (see Definition \ref{defn:tn})   and prove
the following:
\begin{thm*}
  Let $\{A_\fkt~:~\fkt\in\T_n(\ell)\}$ be a set of representatives of
  the orbit of diagrams generating $\overline{J_n}(\ell)$ under
  conjugation by $\sym{n}$, and $D_\fkt$ be the sum of elements in the
  orbit containing $A_\fkt$. Define
  $X_n(\ell)=\sum_{\fkt\in\T_n(\ell)}c_\fkt D_\fkt$ for some scalars
  $c_\fkt$. Then for $\overline{u}\in \overline{J_n}(\ell)$ the
  equation
  \[\overline{u}X_n(\ell)=-\overline{u}\]
is always solvable in the $c_\fkt$. Moreover, setting
$\phi_n(\ell)=1+X_n(\ell)$ gives the splitting idempotent of the short
exact sequence
\[\ses{\overline{J_n}(\ell)}{B_n(\dd)}{B_n(\dd)/\overline{J_n}(\ell)}.\]
\end{thm*}

This paper is structured as follows: In Section \ref{sec:defns} we set
up the definitions for the rest of the paper and classify the
$\sym{n}$-conjugacy classes of elements of $B_n$. In Section
\ref{sec:idemp} we use this to construct the splitting idempotents
related to certain ideals in $B_n$. Section \ref{sec:prim} contains
some background representation theory needed to obtain some of the
primitive central idempotents of $B_n$, and Section \ref{sec:eg}
provides several applications of the theory. Finally, there are two
supplementary sections: one with the splitting idempotents in $B_6$ to
show that the method we obtain does give results that would previously
have been inaccessible, and another comparing the complexity of our
method to a previously known procedure from \cite{leducram} to justify
its use.
\section{Preliminaries}
\label{sec:defns}

\subsection{Young tableaux}
\label{sec:partitions}

Given a partition $\lambda=(\lambda_1,\dots,\lambda_m)$
of $n\in\N$  (i.e. $\sum_i \lambda_i =n$ and $\lambda_1 \ge \lambda_2
\ge \cdots \ge \lambda_m$), 
written $\lambda\vdash n$, also $|\lambda|=n$, we define the
\emph{Young diagram} $[\lambda]$ to be the set
\[[\lambda]=\{(i,j)\in\N^2~|~1\leq i\leq  m,~1\leq j\leq \lambda_i\}.\]
We depict these graphically in the plane by a configuration of boxes
in the ``English convention'', for instance the partition $(5,3,3,1)$
of $12$ has Young diagram
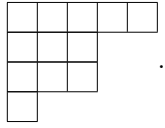
\[\yng(5,3,3,1).\]
We will confuse partitions and their Young diagrams in the usual way.

A \emph{Young tableau} of shape $\lambda$ is a function
\[\fkt:[\lambda]\longrightarrow T,\]
where $T$ is a non-empty set. We can equivalently think of $\fkt$ as a filling of the boxes of $[\lambda]$ by elements of $T$. We will also confuse these two definitions.

\subsection{The Brauer algebra}
\label{sec:brauer}
If $T$ is a finite set of size $2m$ for some $m\in\N$, then write $\J(T)$ for the set of pair partitions of $T$, that is the set
\[\J(T)=\big\{a_1\sqcup\dots\sqcup a_m~|~a_i\subset T,~|a_i|=2\text{ for all }i\big\}\] 
For $n\in\N$ let $\uu{n}=\{1,2,\dots,n\}$, $\uu{n'}=\{1',2',\dots,n'\}$, and define the function
\begin{align}
  \op:\uu{n}\cup\uu{n'}&\longrightarrow\uu{n}\cup\uu{n'}\label{eq:op}\\
  x\in\uu{n}&\longmapsto x'\nonumber\\
  x'\in\uu{n'}&\longmapsto x.\nonumber
\end{align}
Fix an indeterminate $\dd$ and let $R$ be a $\Z[\dd]$-algebra, i.e. a commutative ring with a single parameter. The \emph{Brauer algebra} $B_n=B_n(\dd)$ is the $R$-algebra with basis $\J_n=\J(\uu{n}\cup\uu{n'})$. We can represent any element $A$ of $\J_n$ as a graph in the plane, with vertex set $\uu{n}\cup\uu{n'}$ and an edge between vertices $x$ and $y$ if $\{x,y\}\in A$. We will identify all graph depictions of the same element $A$, and typically draw the vertices as two horizontal rows labelled by $\uu{n}$ and $\uu{n'}$ as in the following example. Note then that $x$ and $\op(x)$ are vertically opposite one another.
\begin{eg}
   Let $A=\big\{\{1,4\},\{2,4'\},\{3,5\},\{6,3'\},\{1',2'\},\{5',6'\}\big\}\in\J_6$. This has the following graphical depiction:
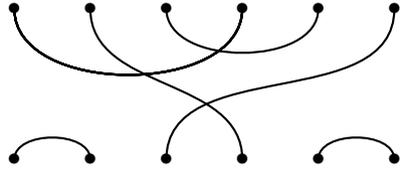
\[\tikz{
    \foreach \x in {1,...,6}
d    {
      \draw (\x,0) node {$\bullet$};
      \draw (\x,2) node {$\bullet$};
    }
    \draw[thick] (1,2) to[out=-90,in=-90] (4,2);
    \draw[thick] (2,2) to[out=-90,in=90] (4,0);
    \draw[thick] (3,2) to[out=-90,in=-90] (5,2);
    \draw[thick] (6,2) to[out=-90,in=90] (3,0);
    \draw[thick] (1,0) to[out=90,in=90] (2,0);
    \draw[thick] (5,0) to[out=90,in=90] (6,0);
  }.\]
\end{eg}
We wish to distinguish edges that connect nodes on the same side of the diagram or opposite sides. To do this we define the \emph{type function} on a pair partition $A\in\J_n$:
\begin{align}
  \tp:A&\longrightarrow\{N,S,P\}\label{eq:tp}\\
  \{x,y\}&\longmapsto\begin{cases}
    N&\text{if }x,y\in\uu{n},\\
    S&\text{if }x,y\in\uu{n'},\\
    P&\text{otherwise}.
\end{cases}\nonumber
\end{align}
We will refer to these three cases as northern horizontal arcs, southern horizontal arcs and propagating lines respectively. This allows us to define the following subsets of $\J_n$:
\begin{align*}
  \J_n[\ell]&=\big\{A\in\J_n~|~\text{$A$ contains precisely $\ell$ components $a_i$ such that $\tp(a_i)=P$}\big\},\text{ and}\\
  \J_n(\ell)&=\bigcup_{m\leq\ell}\J_n[m].
\end{align*}
In other words $\J_n[\ell]$ can be thought of as the set of diagrams with precisely $\ell$ propagating lines, and $\J_n(\ell)$ as the set of diagrams with at most $\ell$ propagating lines.

Multiplication in $B_n$ is defined by vertical concatenation of diagrams. Given $A,B\in\J_n$ we compute $AB$ by drawing $A$ on top of $B$ so that the southern nodes of $A$ and the northern nodes of $B$ coincide pointwise. This defines a new graph $A\circ B$ on three rows of vertices. Let $v(A,B)$ be the number of connected components of $A\circ B$ involving only vertices in the middle row. By considering the connected components of vertices on the top and bottom rows we obtain a pair partition $\pi(A\circ B)$, and define $AB=\dd^{v(A,B)}\pi(A\circ B)$. Note that this multiplication cannot increase the number of propagating lines in a diagram, and hence the set $\J_n(\ell)$ is an $R$-basis of an ideal $J_n(\ell)\subset B_n$.

The Brauer algebra is a unital algebra with identity element
\[1=1_n=\big\{\{1,1'\},\{2,2'\},\dots,\{n,n'\}\big\},\]
and is generated by elements $u_1,u_2,\dots,u_{n-1}$ and $\sigma_1,\sigma_2,\dots,\sigma_{n-1}$, where
\begin{align*}
  u_i&=\left(1_n\backslash\big\{\{i,i'\},\{i+1,(i+1)'\}\big\}\right)\cup\big\{\{i,i+1\},\{i',(i+1)'\}\big\}\\
  \sigma_i&=\left(1_n\backslash\big\{\{i,i'\},\{i+1,(i+1)'\}\big\}\right)\cup\big\{\{i,(i+1)'\},\{i+1,i'\}\big\}.
\end{align*}
We can depict these elements graphically as follows:
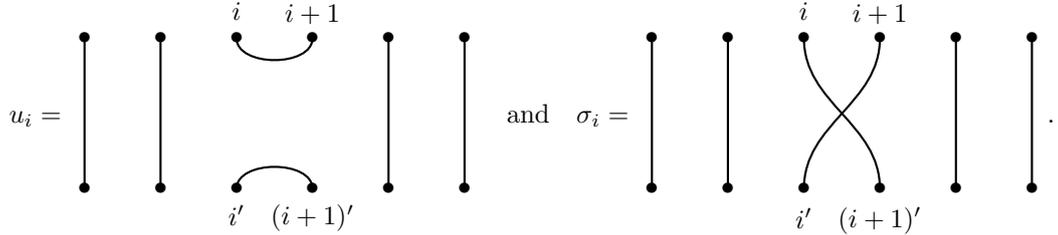
\[u_i=\tikz[baseline=2.5em]{
    \foreach \x in {1,...,6}
    {
      \draw (\x,0) node {$\bullet$};
      \draw (\x,2) node {$\bullet$};
    }
    \draw[thick] (1,2) to[out=-90,in=90] (1,0);
    \draw[thick] (2,2) to[out=-90,in=90] (2,0);
    \draw[thick] (3,2) node[yshift=10pt] {$i$} to[out=-90,in=-90] (4,2) node[yshift=9pt] {$i+1$};
    \draw[thick] (3,0) node[yshift=-10pt] {$i'$} to[out=90,in=90] (4,0) node[yshift=-11pt] {$(i+1)'$};
    \draw[thick] (5,2) to[out=-90,in=90] (5,0);
    \draw[thick] (6,2) to[out=-90,in=90] (6,0);
  }\hspace{1em}\text{and}\hspace{1em}
  \sigma_i=\tikz[baseline=2.5em]{
    \foreach \x in {1,...,6}
    {
      \draw (\x,0) node {$\bullet$};
      \draw (\x,2) node {$\bullet$};
    }
    \draw[thick] (1,2) to[out=-90,in=90] (1,0);
    \draw[thick] (2,2) to[out=-90,in=90] (2,0);
    \draw[thick] (3,2) node[yshift=10pt] {$i$} to[out=-90,in=90] (4,0) node[yshift=-12pt] {$(i+1)'$};
    \draw[thick] (4,2) node[yshift=9pt] {$i+1$} to[out=-90,in=90] (3,0) node[yshift=-11pt] {$i'$};
    \draw[thick] (5,2) to[out=-90,in=90] (5,0);
    \draw[thick] (6,2) to[out=-90,in=90] (6,0);
  }.\]
As mentioned in the introduction, we wish to work over a ring that is amenable to both specialisation of $\dd$ and moving to fields of characteristic $p\geq0$, whilst still allowing us to invert monic polynomials in $\dd$. For our purposes, this ring will be
\[K=\{f/g~|~f,g\in\Z[\dd],\text{ $g$ monic},~\deg(f)\leq\deg(g)\},\]
a subring of $\mathbb{Q}(\dd)$ containing $\Z[\dd^{-1}]$. It is the intersection of the discrete valuation ring $\Z[\dd^{-1}]_{\langle\dd^{-1}\rangle}$ with the localisation of $\Z[\dd]$ at the set of all monic polynomials. The quotient of $K$ by the principal ideal $K\dd^{-1}$ is isomorphic to $\Z$. An element $x\in K$ is a unit in $K$ if and only if $x\equiv\pm1$ mod $K\dd^{-1}$. In order to use this ring, we must substitute the generator $u_i$ by
\begin{equation}\label{eq:ui}
  \overline{u_i}=\frac{1}{\dd}u_i.
\end{equation}
We then view the Brauer algebra as the $K$-algebra generated by the $\sigma_i$ and $\overline{u_i}$. Writing a pair partition $A$ as a product of generators $A=\prod_{j=1}^mA_j$ where $A_j=\sigma_{i_j}$ or $u_{i_j}$ for $1\leq i_j\leq n-1$, we let $\overline{A}=\prod_{j=1}^m\overline{A_j}$, where
\[\overline{A_j}=\begin{cases}
    \sigma_{j_i}&\text{if }A_j=\sigma_{j_i},\\
    \overline{u_{j_i}}&\text{if }A_j=u_{j_i}.
  \end{cases}\]
Then a basis of $B_n$ over $K$ is given by 
\[\overline{\J_n}=\{\overline{A}~|~A\in\J_n\}.\]
We analogously define 
\begin{align*}
  \overline{\J_n}[\ell]&=\{\overline{A}~|~A\in\J_n[\ell]\},\\
  \overline{\J_n}(\ell)&=\{\overline{A}~|~A\in\J_n(\ell)\},\text{ and}\\
  \overline{J_n}(\ell)&=B_n\overline{\J_n}(\ell)B_n.
\end{align*}
Note that since all elements of $\J_n[n]$ are generated by the $\sigma_i$, we have $\overline{\J_n}[n]=\J_n[n]$.
\subsection{Spore function on pair partitions}
\label{sec:spore}
The subalgebra $K\J_n[n]$ of $B_n$ generated by the $\sigma_i$ is
isomorphic to $K\sym{n}$, where permutations are composed
left-to-right. Thus $B_n$ is both a left and a right $K\sym{n}$-module
by restriction. In particular we can conjugate pair partitions by
elements $\sigma\in\sym{n}$, which amounts to relabelling 
the nodes $x$, $x'$ with 
$\sigma x$ and $\sigma x'$. Write
$A^{\sym{n}}$ for the orbit of $A\in\J_n$ under conjugation by
$\sym{n}$. Note that $A\in\J_n[n]$ implies
$A^{\sym{n}}\subset\J_n[n]$ which is in natural bijection with
$\sym{n}$, and there is the usual
observation that conjugacy classes are indexed by integer partitions
of $n$. We also define
\[A_\Sigma=A^{\sym{n}}_\Sigma=\sum_{B\in A^{\sym{n}}}B,\]
the $\sym{n}$-orbit sum of $A$.

The rest of this section is devoted to classifying the $\sym{n}$-conjugacy classes of $\J_n$.

\begin{defn}\label{defn:tn}
Given two Young tableaux $\fks,\fkt$ with entries in $\{N,S,P\}$ and
underlying Young diagram $\lambda=(\lambda_1,\dots,\lambda_m)$ with
$\lambda_m\neq0$, we say $\fks\sim \fkt$ if there exists a permutation
$\sigma\in\sym{m}$ such that the $\sigma i$-th row of $\fkt$ can be
obtained from the $i$-th row of $\fks$ by cycling and/or reversing the
entries. It is clear then that $\sim$ is an equivalence relation. For
$n\in\N$, we let $\T_n$ be the set
  \[\T_n=\big\{\{N,S,P\}^{[\lambda]}~|~\lambda\vdash n\big\}/\sim.\]
Let also $\T_n(\ell)$ be the subset of $\T_n$ containing all tableaux
with at most $\ell$ entries equal to $P$.
\end{defn}

\begin{defn}\label{defn:spore}

  We define the \emph{Spore} function
  \[\Sp:\J_n\longrightarrow\T_n\]
  as follows. For a pair partition $A\in\J_n$ begin by decomposing $A$
  into a disjoint union of non-empty sets
  \[A=A_1\sqcup\dots\sqcup A_m,\]
of maximal possible $m$,
  where for all $a\in A_i,b\in A_j$ $(i\neq j)$, if $x\in a$ then
  $\op(x)\not\in b$, where $\op$ is the function  defined in
  $\eqref{eq:op}$ above. We also relabel so that
  $|A_1|\geq|A_2|\geq\dots\geq|A_m|$. This decomposition defines an
  integer partition $\lambda_A$ of $n$, where $(\lambda_A)_i=|A_i|$.

We now associate a Young tableau $\fks$ of shape $\lambda_A$ to
$A$. For each part $A_i$ of the above decomposition of $A$ we order
the components $a_1,\dots,a_{\lambda_i}$ as follows. We choose $a_1$
arbitrarily, and pick an element $x_1\in a_1$. Now given $a_j$ and
$x_j\in a_j$ we define $a_{j+1}\in A_i$ to be the component containing
$\op(x_j)$ and $x_{j+1}$ to be the element of
$a_{j+1}\backslash\{\op(x_j)\}$. This process ends when we return to
the set $a_1$. Then $A_i$ will have the form
\begin{align*}
  A_i&=\{a_1,a_2,a_3,\dots,a_{(\lambda_A)_i}\}\\
     &=\left\{\{x_1,\op(x_{(\lambda_A)_i})\},\{\op(x_1),x_2\},\{\op(x_2),x_3\},\dots,\{op(x_{(\lambda_A)_i-1}),x_{(\lambda_A)_i}\}\right\}.
\end{align*}
In the $i$-th row of the Young diagram $[\lambda_A]$ we then fill the $j$-th box with the symbol $\tp(a_j)$, where $\tp$ is the type function defined in $\eqref{eq:tp}$ above.
\end{defn}
\begin{prop}
  The function $\Sp$ is well-defined.
\end{prop}
\begin{proof}
  We must show that any pair of tableaux $\fks,\fkt$ constructable from an element $A\in\J_n$ satisfy $\fks\sim \fkt$. In the process described above we make several choices. Firstly, if any of the parts $A_i$ contain the same number of components, we can place the corresponding rows of the Young tableau in any order. However we can obtain any of these tableaux by performing a permutation of the rows, which will give the element $\sigma\in\sym{m}$ from Definition \ref{defn:tn}.

Once we have chosen an order on the parts $A_i$, the next choice is to pick a component $a_1$ and an element $x_1\in a_1$. Choosing a different component for $a_1$ amounts to cycling the sequence of the $a_i$, and choosing a different element $x_1$ reverses the sequence. We therefore see that both tableau represent the same class in $\T_n$.
\end{proof}
\begin{rem}\label{rem:spore}
  The process that defines $\Sp(A)$ does not depend on the actual values of the $x_j\in\uu{n}\cup\uu{n'}$, only the components in which they reside. This is to be expected as we can change values of the $x_j$ by $\sym{n}$-conjugation, and the Spore function is intended to be invariant under this.
\end{rem}
Before we move on to use the Spore function, we provide the following example.
\begin{eg}\label{eg:spore}
  Let $A=\big\{\{1,4\},\{2,4'\},\{3,5\},\{6,3'\},\{1',2'\},\{5',6'\}\big\}\in\J_6$. This decomposes into $A=A_1\sqcup A_2$, where
  \begin{align*}
    A_1&=\big\{\{1,4\},\{2,4'\},\{1',2'\}\big\},\text{ and}\\
    A_2&=\big\{\{3,5\},\{6,3'\},\{5',6'\}\big\}.
  \end{align*}
Starting with the first element of the first pair as written above, we obtain sequences $(N,P,S)$ and $(N,S,P)$ for $A_1$ and $A_2$ respectively. Therefore
\[\Sp(A)=\young(NSP,NPS).\]
\end{eg}
There is an alternative way of constructing the Spore function using the diagram form of the Brauer algebra. Given the diagram of a pair partition $A\in\J_n$, begin by labelling all northern horizontal arcs by $N$, southern horizontal arcs by $S$ and propagating arcs $P$. Then identify all pairs of nodes $i,i'$ for $1\leq i\leq n$. The resulting diagram has $n$ nodes connected by a series of arcs each labelled $N$, $S$ or $P$, such that each node has valency $2$. The connected components of this diagram then partition the set of nodes. These components then define an integer partition $\lambda$ of $n$, where $\lambda_i$ is the number of nodes in the $i$-th largest connected component. For each $i$, we choose a node in the $i$-th largest component and a direction, walk around this component and record the sequence of edge labels we encounter in the $i$-th row of the Young diagram $[\lambda]$. This defines a Young tableaux $\fkt$ with entries in $\{N,S,P\}$. We therefore set $\Sp(A)=\fkt\in\T_n$.
\begin{eg}
  Let $A=\big\{\{1,4\},\{2,4'\},\{3,5\},\{6,3'\},\{1',2'\},\{5',6'\}\big\}\in\J_6$ as before. We draw the diagram and label the edges $N$, $S$ or $P$ below.
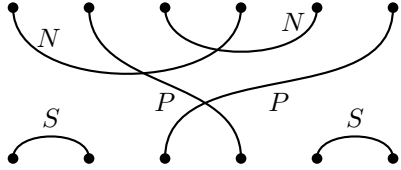
\[\tikz{
    \foreach \x in {1,...,6}{
      \fill[black] (\x,0) circle (2pt);
      \fill[black] (\x,2) circle (2pt);}
    \draw[thick] (1,2) to[out=-90,in=-90] node[pos=0.25,above] {$N$} (4,2);
    \draw[thick] (2,2) to[out=-90,in=90] node[pos=0.5,below] {$P$} (4,0);
    \draw[thick] (3,2) to[out=-90,in=-90] node[pos=0.75,above] {$N$} (5,2);
    \draw[thick] (6,2) to[out=-90,in=90] node[pos=0.5,below] {$P$} (3,0);
    \draw[thick] (1,0) to[out=90,in=90] node[pos=0.5,above] {$S$} (2,0);
    \draw[thick] (5,0) to[out=90,in=90] node[pos=0.5,above] {$S$} (6,0);
  }\]
After identifying opposite pairs of nodes we have the following diagram.
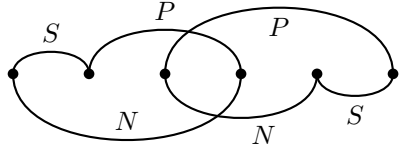
\[\tikz{
    \foreach \x in {1,...,6}{\fill[black] (\x,0) circle (2pt);}
    \draw[thick] (1,0) to[out=-90,in=-90] node[pos=0.5,above] {$N$} (4,0);
    \draw[thick] (2,0) to[out=90,in=90] node[pos=0.5,above] {$P$} (4,0);
    \draw[thick] (3,0) to[out=-90,in=-90] node[pos=0.6,below] {$N$} (5,0);
    \draw[thick] (6,0) to[out=90,in=90] node[pos=0.5,below] {$P$} (3,0);
    \draw[thick] (1,0) to[out=90,in=90] node[pos=0.5,above] {$S$} (2,0);
    \draw[thick] (5,0) to[out=-90,in=-90] node[pos=0.5,below] {$S$} (6,0);
  }\]
We see that we have two connected components, each containing three nodes. Starting with the leftmost node in each part and walking counter-clockwise around we record the same tableaux as in Example \ref{eg:spore}.
  \begin{align*}
    A_1&=\big\{\{1,4\},\{2,4'\},\{1',2'\}\big\},\text{ and}\\
    A_2&=\big\{\{3,5\},\{6,3'\},\{5',6'\}\big\}.
  \end{align*}
Starting with the first element of the first pair as written above, we obtain sequences $(N,P,S)$ and $(N,S,P)$ for $A_1$ and $A_2$ respectively. Therefore
\[\Sp(A)=\young(NSP,NPS).\]
\end{eg}
\begin{prop}\label{prop:spore}
  For all $A,B\in\J_n$, $A^{\sym{n}}=B^{\sym{n}}$ if and only if $\Sp(A)=\Sp(B)$.
\end{prop}
\begin{proof}
  The effect of conjugation by an element of the symmetric group on a diagram $A$ is to apply the same permutation to the the set $\uu{n}$ and $\uu{n'}$. Therefore when we decompose $A=A_1\sqcup\dots\sqcup A_m$, neither the size of the $A_i$ nor the type of the component parts is affected. We therefore have that for all $A\in\J_n$ and $\sigma\in\sym{n}$,
  \[\Sp(A)=\Sp(\sigma A\sigma^{-1}).\]
It follows that $A^{\sym{n}}=B^{\sym{n}}$ implies $\Sp(A)=\Sp(B)$.

Now assume that $\Sp(A)=\Sp(B)$. We saw in Remark \ref{rem:spore} that the labels of the $x_j$ do not matter, so we may assume that the first components $A_1$ and $B_1$ contains elements $x,x'$ for $1\leq x\leq(\lambda_A)_1$, the second components contain $x,x'$ for $(\lambda_A)_1+1\leq x\leq(\lambda_A)_2$ and so on. Since the two diagrams then are formed of disjoint components of corresponding sizes, we may assume that there is only one part in the decomposition $A=A_1$ (hence also $B=B_1$). Then it must be possible to write $A=\{a_1,\dots,a_n\}$ and $B=\{b_1,\dots,b_n\}$ such that $\tp(a_i)=\tp(b_i)$ for all $i$. We also have sequences of distinct elements $x_i\in a_i$ and $y_i\in b_i$ such that (ignoring primes) all values $1,\dots,n$ appear in each sequence. We then construct a permutation $\sigma\in\sym{n}$ by setting $\sigma x_i=y_i$ for all $i$. Hence we have $A=\sigma B\sigma^{-1}$, and therefore $A^{\sym{n}}=B^{\sym{n}}$.
\end{proof}
\begin{prop}
  For each $\ell$, the image of $\J_n[\ell]$ under $\Sp$ is the set of equivalence classes in $\T_n$ whose representative tableaux have the following properties:
  \begin{itemize}
  \item there are the same number of $N$ and $S$ in each row;
  \item ignoring the $P$, the $N$ and $S$ alternate across each row;
  \item $P$ appears $\ell$ times across the whole tableau.
  \end{itemize}
\end{prop}
\begin{proof}
  For the first property, note that for each horizontal arc on the top of each component of the diagram we must also have a horizontal arc on the bottom.

  The second property follows from the fact that each node in the original diagram is connected to precisely one edge, so we cannot have successive arcs at the top since this will require a node of valency two between them (and similarly for the bottom).

  The last property is by definition of $\J_n[\ell]$, as this states that the original diagram has precisely $\ell$ propagating lines.
\end{proof}
\begin{rem*}
  Note that the tableaux in the image of $\J_n[n]$ have all entries equal to $P$, and are therefore in bijection with the set of partitions $\lambda$ of $n$. This is to be expected, as $\J_n[n]$ is isomorphic as a group to $\sym{n}$.
\end{rem*}
\section{Construction of the splitting idempotent}
\label{sec:idemp}
As outlined in the introduction, we will follow the approach of \cite{martinwoodcock}. This relies on the following lemma:
\begin{lem}[{\cite[Section 1]{martinwoodcock}}]\label{lem:splitting}
 Let $J\subset \Lambda$ be an ideal of a unital algebra $\Lambda$, then the short exact sequence of $\Lambda$-bimodules
\[\ses{J}{\Lambda}{\Lambda/J}\]
splits if and only if there is an element $\phi_J\in \Lambda$ with the following properties:
\begin{enumerate}
\item $\phi_J\equiv1_\Lambda$ mod $J$;
\item $\phi_JJ=J\phi_J=0$.
\end{enumerate}
If $\phi_J$ exists then it is the unique idempotent with these properties, and moreover $\phi_J\in Z(\Lambda)$, the centre of $\Lambda$. 
\end{lem}
 
For $\Lambda'\subset\Lambda$ a subalgebra (or indeed any subset), define $Z_{\Lambda'}(\Lambda)$ as the set of elements of $\Lambda$ that commute with $\Lambda'$. Obviously $Z(\Lambda)\subset Z_{\Lambda'}(\Lambda)$. Thus we can start to search for elements of $Z(\Lambda)$ by looking for elements of $Z_{\Lambda'}(\Lambda)$.

We will then examine $Z_{K\sym{n}}(B_n)$, where $K\sym{n}$ is the subalgebra of $B_n$ with basis $\J_n[n]$. We are therefore interested in elements of $\overline{\J_n}$ that are invariant under conjugation by all elements of $\sym{n}$. Consider an element $x\in Z_{K\sym{n}}(B_n)$ of the form
\begin{align*}
  x&=\sum_{A\in\overline{\J_n}}c_AA\hspace{2em}(c_A\in K)\\
   &=\sigma x\sigma^{-1}\\
   &=\sum_{A\in\overline{\J_n}}c_A\sigma A\sigma^{-1}\\
   &=\sum_{A\in\overline{\J_n}}c_{\sigma^{-1}A\sigma}A,
\end{align*}
where we have used the fact that conjugation by $\sigma\in\sym{n}$ is a permutation on $\overline{\J_n}$. Thus $x\in Z_{K\sym{n}}(B_n)$ implies $c_A=c_{\sigma A\sigma^{-1}}$ for all $\sigma$. Evidently for any $A$,
\[\sum_{\sigma\in\sym{n}}\sigma A\sigma^{-1}\in Z_{K\sym{n}}(B_n).\]
In characteristic zero, all possible multiplicities in this sum are units, so $Z_{\sym{n}}(B_n)$ has a basis of elements of this form. However we wish to find a basis valid in arbitrary characteristic.

\begin{lem}\label{lem:basis}
  For each $\fkt\in\Im(\Sp)\subset\T_n$, let $A_\fkt\in\J_n$ be any pair partition such that $\Sp(A_\fkt)=\fkt$. Writing $D_\fkt=(A_\fkt)_\Sigma$, the set
  \[\{D_\fkt~|~\fkt\in\Im(\Sp)\}\]
is a basis of $Z_{\sym{n}}(B_n)$.
\end{lem}
\begin{proof}
  It is clear that the elements $A_\Sigma$ $(A\in\J_n)$ span this space, and from Proposition \ref{prop:spore} we see that $A_\Sigma=B_\Sigma$ if and only if $\Sp(A)=\Sp(B)$. The result follows.
\end{proof}

Recall from Section \ref{sec:brauer} that $\overline{J_n}(\ell)$ is the ideal of $B_n$ with basis $\overline{\J_n}(\ell)$, and for $\ell<n$ we write $\phi_n(\ell)$ for the corresponding splitting idempotent in the sense of Lemma \ref{lem:splitting}. We will see below that this idempotent exists for our chosen ring $K$. Define $X_n(\ell)$ by $\phi_n(\ell)=1+X_n(\ell)$. Since $X_n(\ell)$ is central, and hence in $Z_{\sym{n}}(B_n)$, we have
\[X_n(\ell)=\sum_{\fkt\in\T_n(\ell)}c_\fkt D_\fkt\]
where the scalars $c_\fkt$ are to be determined. By Lemma \ref{lem:splitting}(ii) a necessary condition is given by $dX_n(\ell)=-d$ for $d\in\overline{J_n}(\ell)$. Thus in particular for $\overline{u}=\overline{u_1}\,\overline{u_3}\,\overline{u_5}\dots \overline{u_{n-\ell-1}}$ (where the $\overline{u_i}$ are as in \eqref{eq:ui}) a necessary condition is 
\begin{equation}
\overline{u}X_n(\ell)=-\overline{u}.\label{eq:necessary}
\end{equation}
We will use this equation to obtain several linear equations in the $c_\fkt$, show that these are linearly independent and hence solve to obtain the values of $c_\fkt$.

 We may assume that $A_\fkt$ has an arc between nodes $2j+1$ and $2j+2$ for $j=0,1,\dots,\frac{1}{2}(n-\ell-2)$. Then $\overline{u}A_\fkt=A_\fkt$ for all $\fkt\in\T_n(\ell)$. Moreover the following proposition shows that this relation is uniquely satisfied by the action of $u$ on $A_\fkt$.
\begin{prop}
  Suppose $A\neq A_\fkt$ satisfies $\overline{u}A=\dd^rA_\fkt$ for some $r\in\Z$. Then $r<0$.
\end{prop}
\begin{proof}
  Clearly $r=0$ is the maximum possible power of $\dd$ since we are working in the ring $K$. So we prove that if this maximum is attained, then $A=A_\fkt$. 

Firstly, it is clear that if $r=0$, then we must cancel each factor $\frac{1}{\dd}$ from each of the $\overline{u_i}$ constituting $\overline{u}$ by forming closed loops. Thus nodes $2j+1$ and $2j+2$ must be joined for $j=0,1,\dots,\frac{1}{2}(n-\ell-2)$. Next, the action of $\overline{u}$ cannot change the arrangement of any southern arcs, and it acts as the identity on the remaining $\ell$ propagating or northern arcs. Clearly this implies that if $\overline{u}A=\dd^rA_\fkt$ with $r$ maximal, then $A=A_\fkt$.
\end{proof}
Writing $\T_n(\ell)=\{\fkt_1,\dots,\fkt_m\}$ with $\Sp(u)=\fkt_1$ we have a system of equations
\begin{equation}
\left(
\begin{array}{cccc}
  p_1^{(1)}(\dd)&p_2^{(1)}(\dd)&\cdots&p_{m}^{(1)}(\dd)\\
  p_1^{(2)}(\dd)&p_2^{(2)}(\dd)&\cdots&p_{m}^{(2)}(\dd)\\
  \vdots&\vdots&\ddots&\vdots\\
  p_1^{(m)}(\dd)&p_2^{(m)}(\dd)&\cdots&p_{m}^{(m)}(\dd)
\end{array}
\right)
\left(
\begin{array}{c}
  c_{\fkt_1}\\
  c_{\fkt_2}\\
  \vdots\\
  c_{\fkt_m}
\end{array}
\right)=
\left(
\begin{array}{c}
  -\dd^{-\frac{1}{2}(n-\ell)}\\
  0\\
  \vdots\\
  0
\end{array}
\right)\label{eq:matrix}\end{equation}
where the $p_j^{(k)}$ are elements of $K$, $p_j^{(j)}\equiv1$ mod $K\dd^{-1}$, and $p_j^{(k)}\in K\dd^{-1}$ for $k\neq j$.
Therefore the determinant of this matrix is also an element of $K$ with leading term $1$, which is generically non-zero. In order to invert this matrix, it may be true that we have to work over the field of rational polynomials in $\dd$. However the following proposition shows that this is not the case.
\begin{prop}
  The coefficients $c_{\fkt_i}$ all lie in $K$.
\end{prop}
\begin{proof}
  Denote by $M$ the $m\times m$ matrix in \eqref{eq:matrix}. Since $\det M\equiv1$ mod $K\dd^{-1}$, it is a unit in $K$. An application of Cramer's rule then shows that
\[c_{\fkt_i}=\frac{\det M_i}{\det M},\]
 where $M_i$ is the matrix obtained by replacing the $i$-th column of $M$ by $(-\dd^{-\frac{1}{2}(n-\ell)},0,\dots,0)^T$. Therefore we must show that $\det M_i$ is an element of $K$. But in the construction of $M_i$ we are simply replacing the $p_i^{(j)}(\dd)$ by either $0$ or $-\dd^{-\frac{1}{2}(n-\ell)}$, which are also elements of $K$. Therefore $\det M_i\in K$, and the result follows.
\end{proof}
The above proves the main theorem of this paper, that the condition \eqref{eq:necessary} is also sufficient.
\begin{thm}\label{thm:main}
  For each $\fkt\in\T_n(\ell)$ let $D_\fkt=(A_\fkt)_\Sigma$, where $A_\fkt\in\J_n$ is any pair partition such that $Sp(A_\fkt)=\fkt$. Setting $X_n(\ell)=\sum_{\fkt\in\T_n(\ell)}c_\fkt D_\fkt$ for some $c_\fkt\in K$ and $\overline{u}=\overline{u_1}\,\overline{u_3}\dots\overline{u_{n-\ell-1}}$, the equation
  \[\overline{u}X_n(\ell)=-\overline{u}\]
is always solvable in the $c_\fkt$. Moreover, by defining $\phi_n(\ell)=1+X_n(\ell)$ we obtain the splitting idempotent corresponding to the short exact sequence
\[\ses{\overline{J_n}(\ell)}{B_n}{B_n/\overline{J_n}(\ell)}.\]
\end{thm}

\section{Representation theory and primitive central idempotents}
\label{sec:prim}

In this section we will study the Brauer algebra over a field of characteristic zero, which for our purposes amounts to extending scalars of the ring $K$ to $\Q\otimes_\Z K$ and specialising $\dd$ to an element of $\Z$. We will assume some familiarity with the representation theory of the Brauer algebra over a field (see for instance \cite{cdm}, \cite{cmpx}, \cite{grahamlehrer}, \cite{rui}). In particular, the algebra $B_n$ is cellular \cite{grahamlehrer}, and thus comes equipped with cell modules. These cell modules are indexed by integer partitions of $n,n-2,\dots,0/1$, and generically so too are the simple modules. Write $\Delta_n(\lambda)$ (resp. $L_n(\lambda)$) for the cell (resp. simple) module indexed by the partition $\lambda$.

The family $B_n$ $(n\geq0)$ of Brauer algebras form a tower of recollement, in the sense of \cite{cmpx}. We therefore have a family of localisation functors
\[F_n:B_n\textbf{-mod}\rightarrow B_{n-2}\textbf{-mod}\]
and globalisation functors
\[G_n:B_n\textbf{-mod}\rightarrow B_{n+2}\textbf{-mod}.\]
For all $n\geq0$ and $B_n$-modules $M$, we have $F_{n+2}G_n(M)\cong M$, and each $G_n$ is a full embedding. Moreover for all partitions $\lambda\vdash n,n-2,\dots,0/1$,
\begin{align*}
  F_n(\Delta_n(\lambda))&\cong\begin{cases}
    \Delta_{n-2}(\lambda)&\text{if }\lambda\vdash n-2,n-4,\dots,0/1\\
    0&\text{if }\lambda\vdash n,\text{ and}
    \end{cases}\\
  G_n(\Delta_n(\lambda))&\cong\Delta_{n+2}(\lambda).
\end{align*}
In the generic case over a field of characteristic zero or $p>n$ the Brauer algebra is semisimple, and the cell modules are both simple and indecomposable projective, so are generated by a primitive central idempotent $\phi_n(\lambda)$. Therefore $\phi_n(\ell)$ decomposes into a sum of $\phi_n(\lambda)$ where $\lambda$ is a partition of $\lambda\vdash\ell+2,\ell+4,\dots,n$. For $\ell+2<n$ this decomposition is not always easily obtained. However when $\ell=n-2$ we have the following:
\begin{lem}\label{lem:primidemps}
  For $\lambda\vdash n$,
\[\phi_n(\lambda)=\phi_n(n-2)e_\lambda,\]
where $e_\lambda$ is the idempotent in $\Q\sym{n}$ corresponding to the Specht module $S^\lambda$, viewed as an element of $B_n$.
\end{lem}
\begin{proof}
  We show that the action of $B_n$ on the module generated by $\phi_n(n-2)e_\lambda$ is the same as that on the cell module $\Delta_n(\lambda)$. In the case of the latter, all elements with fewer than $n$ propagating lines act as zero, and the remaining act as they would on the Specht module $S^\lambda$.

Since $\phi_n(n-2)$ is central, we need not worry about the order of multiplication above. Now from Lemma \ref{lem:splitting}, we see that $\phi_n(n-2)$ acts as zero on any element with fewer than $n$ propagating lines, and as the identity on the rest. Since all that remains are elements with $n$ propagating lines, they then act on the idempotent $e_\lambda$ as they do in the Specht module $S^\lambda$, proving the lemma.
\end{proof}
When specialising $\dd$ or moving to a field of characteristic $p>0$, it is possible that the Brauer algebra may no longer be semisimple. This will be reflected in the idempotents $\phi_n(\ell)$ and $\phi_n(\lambda)$. Indeed, some of these may no longer be well defined, and will need to be added together in order to clear any singularities. This corresponds to having a non-trivial block in the algebra.

Note first that the denominators in $\phi_n(\ell)$ are all monic polynomials in $\Z[\dd]$, and so are well defined in all characteristics. Assume then that we are working in a field of characteristic zero. Rui's semisimplicity criterion \cite[Theorem 1.2]{rui} tells us that these denominators will vanish when $\dd$ is an element of a certain subset of the integers.

Continuing with the characteristic zero case, suppose $\lambda\vdash n$. If the denominators appearing in $\phi_n(\lambda)$ do not vanish at a chosen value of $\dd\in K$ then the cell module $\Delta_n(\lambda)$ is equal to the simple module $L_n(\lambda)$ and there is a corresponding idempotent in $B_n$ splitting
\[\ses{\mathrm{Ann}(L_n(\lambda))}{B_n}{B_n/\mathrm{Ann}(L_n(\lambda))}.\]
Thus there can be no map $L_n(\lambda)\hookrightarrow\Delta_n(\mu)$ for any partition $\mu\neq\lambda$. Equivalently, if a denominator does vanish then there is a corresponding map. Moreover, if $m$ is the largest propagating number among the elements with diverging coefficients, then $\mu\vdash m$.

We can use globalisation and localisation to overcome the difficulty of computing the $\phi_n(\lambda)$ for $\lambda\vdash \ell<n$. Indeed, due to the cellular structure of $B_n$, if $L_n(\lambda)$ appears as a composition factor of any $\Delta_n(\mu)$ then $|\mu|\leq|\lambda|$. Therefore by localising to $B_\ell$, we do not lose any data about which modules $L_n(\lambda)$ appears in. We will make use of this in the examples in the next section.
\section{Examples}
\label{sec:eg}

Given their links to representation theory (cf. \cite[Chapter 1]{benson} for example), it should not be surprising that in many cases the calculation of central idempotents is a highly non-trivial task, see for instance Murphy's construction of central idempotents in the symmetric group \cite{murphy}. The method described above gives us a general process that, given enough time, will produce central idempotents of $B_n$ and from there some of the primitive central idempotents. In low ranks it is even possible to calculate the $\phi_n(\ell)$ (and some of the $\phi_n(\lambda)$) explicitly by hand. We will do this for $n\leq4$, but for the sake of brevity will suppress many of the details. Instead we will refer to several features of the idempotents that can be interpreted in a representation theoretic manner.

\subsection{Splitting idempotents}\label{sec:exsplit}
Our first task will be to calculate the idempotent $\phi_2(0)$ in $B_2$. By Lemma \ref{lem:basis}, a basis of $Z_{\sym{2}}(B_2)$ is indexed by the tableaux
\[\fks^{(0)}_1=\young(NS),\hspace{1em}\fks^{(2)}_1=\young(PP),\hspace{1em}\text{and }\fks^{(2)}_2=\young(P,P).\]
The tableau corresponding to diagrams with no propagating lines is $\fks^{(0)}_1$, and so
\[X_2(0)=a_{\fks^{(0)}_1}D_{\fks^{(0)}_1},\]
where $D_{\fks^{(0)}_1}=u_1\in B_2$. Theorem \ref{thm:main} requires $\overline{u_1}X_2(0)=-\overline{u_1}$, which is satisfied by setting $a_{\fks^{(0)}_1}=-\frac{1}{\dd}$. Therefore
\[\phi_2(0)=1-\frac{1}{\dd}u_1.\]

We invite the reader to calculate the idempotent $\phi_3(1)$ in $B_3$, as we will not make use of it in the rest of this paper. The case $n=4$ will outline the method and provide enough detail to omit the $n=3$ case.

We now calculate the idempotents $\phi_4(0)$ and $\phi_4(2)$ in $B_4$. This requires us to first find a basis of $Z_{\sym{4}}(B_4)$, which again by Lemma \ref{lem:basis} above is indexed by the tableaux
\[\fkt^{(0)}_1=\young(NSNS),\hspace{1em}\fkt^{(0)}_2=\young(NS,NS),\]
\[\fkt^{(2)}_1=\young(NSPP),\hspace{1em}\fkt^{(2)}_2=\young(NPSP),\hspace{1em}\fkt^{(2)}_3=\young(NSP,P),\hspace{1em}\fkt^{(2)}_4=\young(NS,PP),\hspace{1em}\fkt^{(2)}_5=\young(NS,P,P),\]
\[\fkt^{(4)}_1=\young(PPPP),\hspace{1em}\fkt^{(4)}_2=\young(PPP,P),\hspace{1em}\fkt^{(4)}_3=\young(PP,PP),\hspace{1em}\fkt^{(4)}_4=\young(PP,P,P),\hspace{1em}\text{and }\fkt^{(4)}_5=\young(P,P,P,P).\]
Considering first diagrams with no propagating lines, i.e. those tableaux $\fkt^{(i)}_j$ with $i=0$, we have
\[X_4(0)=b_{\fkt^{(0)}_1}D_{\fkt^{(0)}_1}+b_{\fkt^{(0)}_2}D_{\fkt^{(0)}_2},\]
where
\begin{align*}
  D_{\fkt^{(0)}_1}&=\tikz[scale=0.4,baseline=1em]{
\foreach \x in {1,...,4}{
  \draw (\x,0) node {$\bullet$};
  \draw (\x,2) node {$\bullet$};}
  \draw[thick] (1,2) to[out=-90,in=-90] (2,2);
  \draw[thick] (3,2) to[out=-90,in=-90] (4,2);
  \draw[thick] (1,0) to[out=90,in=90] (3,0);
  \draw[thick] (2,0) to[out=90,in=90] (4,0);}+
\tikz[scale=0.4,baseline=1em]{
\foreach \x in {1,...,4}{
  \draw (\x,0) node {$\bullet$};
  \draw (\x,2) node {$\bullet$};}
  \draw[thick] (1,2) to[out=-90,in=-90] (2,2);
  \draw[thick] (3,2) to[out=-90,in=-90] (4,2);
  \draw[thick] (1,0) to[out=90,in=90] (4,0);
  \draw[thick] (2,0) to[out=90,in=90] (3,0);}+
\tikz[scale=0.4,baseline=1em]{
\foreach \x in {1,...,4}{
  \draw (\x,0) node {$\bullet$};
  \draw (\x,2) node {$\bullet$};}
  \draw[thick] (1,2) to[out=-90,in=-90] (3,2);
  \draw[thick] (2,2) to[out=-90,in=-90] (4,2);
  \draw[thick] (1,0) to[out=90,in=90] (2,0);
  \draw[thick] (3,0) to[out=90,in=90] (4,0);}+
\tikz[scale=0.4,baseline=1em]{
\foreach \x in {1,...,4}{
  \draw (\x,0) node {$\bullet$};
  \draw (\x,2) node {$\bullet$};}
  \draw[thick] (1,2) to[out=-90,in=-90] (3,2);
  \draw[thick] (2,2) to[out=-90,in=-90] (4,2);
  \draw[thick] (1,0) to[out=90,in=90] (4,0);
  \draw[thick] (2,0) to[out=90,in=90] (3,0);}+
 \tikz[scale=0.4,baseline=1em]{
\foreach \x in {1,...,4}{
  \draw (\x,0) node {$\bullet$};
  \draw (\x,2) node {$\bullet$};}
  \draw[thick] (1,2) to[out=-90,in=-90] (4,2);
  \draw[thick] (3,2) to[out=-90,in=-90] (2,2);
  \draw[thick] (1,0) to[out=90,in=90] (2,0);
  \draw[thick] (3,0) to[out=90,in=90] (4,0);}+
\tikz[scale=0.4,baseline=1em]{
\foreach \x in {1,...,4}{
  \draw (\x,0) node {$\bullet$};
  \draw (\x,2) node {$\bullet$};}
  \draw[thick] (1,2) to[out=-90,in=-90] (4,2);
  \draw[thick] (3,2) to[out=-90,in=-90] (2,2);
  \draw[thick] (1,0) to[out=90,in=90] (3,0);
  \draw[thick] (2,0) to[out=90,in=90] (4,0);}\text{ and}\\
D_{\fkt^{(0)}_2}&=\tikz[scale=0.4,baseline=1em]{
\foreach \x in {1,...,4}{
  \draw (\x,0) node {$\bullet$};
  \draw (\x,2) node {$\bullet$};}
  \draw[thick] (1,2) to[out=-90,in=-90] (2,2);
  \draw[thick] (3,2) to[out=-90,in=-90] (4,2);
  \draw[thick] (1,0) to[out=90,in=90] (2,0);
  \draw[thick] (3,0) to[out=90,in=90] (4,0);}+
\tikz[scale=0.4,baseline=1em]{
\foreach \x in {1,...,4}{
  \draw (\x,0) node {$\bullet$};
  \draw (\x,2) node {$\bullet$};}
  \draw[thick] (1,2) to[out=-90,in=-90] (3,2);
  \draw[thick] (2,2) to[out=-90,in=-90] (4,2);
  \draw[thick] (1,0) to[out=90,in=90] (3,0);
  \draw[thick] (2,0) to[out=90,in=90] (4,0);}+
\tikz[scale=0.4,baseline=1em]{
\foreach \x in {1,...,4}{
  \draw (\x,0) node {$\bullet$};
  \draw (\x,2) node {$\bullet$};}
  \draw[thick] (1,2) to[out=-90,in=-90] (4,2);
  \draw[thick] (3,2) to[out=-90,in=-90] (2,2);
  \draw[thick] (1,0) to[out=90,in=90] (4,0);
  \draw[thick] (2,0) to[out=90,in=90] (3,0);}.
\end{align*}
Setting $\overline{u}=\overline{u_1}\,\overline{u_3}$ and requiring $\overline{u}X_4(0)=-\overline{u}$, we obtain the system of equations
\[\left(\begin{array}{cc}
          1+\dd^{-1}&\dd^{-1}\\
          2\dd^{-1}&1
        \end{array}\right)
\left(\begin{array}{c}
        b_{\fkt^{(0)}_1}\\
        b_{\fkt^{(0)}_2}
      \end{array}\right)=
\left(\begin{array}{c}
        0\\
        -\dd^{-2}
        \end{array}\right).\]
Solving this gives
\[b_{\fkt^{(0)}_1}=\frac{1}{\dd(\dd+2)(\dd-1)}\hspace{1em}\text{and}\hspace{1em}b_{\fkt^{(0)}_2}=-\frac{\dd+1}{\dd(\dd+2)(\dd-1)},\]
and we have $\phi_4(0)=1+b_{\fkt^{(0)}_1}D_{\fkt^{(0)}_1}+b_{\fkt^{(0)}_2}D_{\fkt^{(0)}_2}$.

We will now consider diagrams with at most 2 propagating lines, i.e. those tableaux $t^{(i)}_j$ with $i=0,2$, so that
\[X_4(2)=\sum_{i=1}^5c_{\fkt^{(2)}_i}D_{\fkt^{(2)}_i}+\sum_{i=1}^2c_{\fkt^{(0)}_i}D_{\fkt^{(0)}_i}.\]
This time we set $\overline{u}=\overline{u_1}$ and obtain the following system of $7$ linearly independent equations.
\[\left(\begin{array}{ccccccc}
          1+\dd^{-1}&\dd^{-1}&2\dd^{-1}&0&2\dd^{-1}&0&0\\
          2\dd^{-1}&1&0&2\dd^{-1}&0&\dd^{-1}&\dd^{-1}\\
          0&0&1+\dd^{-1}&\dd^{-1}&\dd^{-1}&\dd^{-1}&0\\
          0&0&2\dd^{-1}&1&2\dd^{-1}&0&0\\
          0&0&\dd^{-1}&\dd^{-1}&1+\dd^{-1}&0&\dd^{-1}\\
          0&0&4\dd^{-1}&0&0&1&0\\
          0&0&0&0&4\dd^{-1}&0&1
\end{array}\right)
\left(\begin{array}{c}
        c_{\fkt^{(0)}_1}\\
        c_{\fkt^{(0)}_2}\\
        c_{\fkt^{(2)}_1}\\
        c_{\fkt^{(2)}_2}\\
        c_{\fkt^{(2)}_3}\\
        c_{\fkt^{(2)}_4}\\
        c_{\fkt^{(2)}_5}
\end{array}\right)=
\left(\begin{array}{c}
        0\\
        0\\
        0\\
        0\\
        0\\
        0\\
        -\dd^{-1}
\end{array}\right).\]
Upon solving this we see that
\begin{align*}
  c_{\fkt^{(0)}_1}&=-\frac{3\dd+2}{(\dd-2)(\dd-1)(\dd+2)(\dd+4)}\\
  c_{\fkt^{(0)}_2}&=\frac{\dd^2+3\dd+6}{(\dd-2)(\dd-1)(\dd+2)(\dd+4)}\\
  c_{\fkt^{(2)}_1}&=-\frac{1}{(\dd-2)(\dd+2)(\dd+4)}\\
  c_{\fkt^{(2)}_2}&=-\frac{2}{\dd(\dd-2)(\dd+4)}\\
  c_{\fkt^{(2)}_3}&=\frac{\dd+3}{(\dd-2)(\dd+2)(\dd+4)}\\
  c_{\fkt^{(2)}_4}&=\frac{4}{\dd(\dd-2)(\dd+2)(\dd+4)}\\
  c_{\fkt^{(2)}_5}&=-\frac{\dd^3+4\dd^2-4}{\dd(\dd-2)(\dd+2)(\dd+4)},
\end{align*}
and $\displaystyle\phi_4(2)=1+\sum_{i=1}^5c_{\fkt^{(2)}_i}D_{\fkt^{(2)}_i}+\sum_{i=1}^2c_{\fkt^{(0)}_i}D_{\fkt^{(0)}_i}$.
\begin{rem*}
  Note that the values of $\dd$ for which $\phi_4(0)$ and $\phi_4(2)$ are well-defined coincide with the values of $\dd$ for which $B_4(\dd)$ is semisimple over a field of characteristic zero (see \cite{rui}).
\end{rem*}

\subsection{Connections with representation theory}

We begin by studying the case $n=2$. From \cite{murphy}, the primitive central idempotents in $\Q\sym{2}$ are
\begin{align*}
  e_{(2)}&=\frac{1}{2}(1+(1~2))\text{, and}\\
  e_{(1^2)}&=\frac{1}{2}(1-(1~2)).
\end{align*}
By Lemma \ref{lem:primidemps}, we then have
\begin{align*}
  \phi_2((2))&=\phi_2(0)e_{(2)}\\
             &=\frac{1}{2}(1+(1~2))-\frac{1}{\dd}u_1\\
             &=\frac{1}{2}(D_{\fks^{(2)}_2}+D_{\fks^{(2)}_1})-\frac{1}{\dd}D_{\fks^{(0)}_1}\\
  \phi_2((1^2))&=\phi_2(0)e_{(1^2)}\\
             &=\frac{1}{2}(1-(1~2))\\
             &=\frac{1}{2}(D_{\fks^{(2)}_2}-D_{\fks^{(2)}_1}).
\end{align*}
When $\dd=0$, the coefficient of $\phi_2((2))$ corresponding to elements with zero propagating lines diverges, indicating a non-zero homomorphism $L_2((2))\hookrightarrow\Delta_2(\emptyset)$.

Moving now to the $n=4$ case, note first that we can globalise the $n=2$ case and see that when $\dd=0$, we have a non-zero homomorphism
\[\Delta_4(2)\rightarrow\Delta_4(\emptyset)/M,\]
where $M\subset\Delta_4(\emptyset)$ is a submodule.  

We now calculate the idempotent $e_\lambda\in \Q\sym{4}$ with $\lambda=(3,1)$ using the results of \cite{murphy}: 
\begin{align*}
  e_{(3,1)}&=\frac{3}{8}+\frac{1}{8}(1~2)_\Sigma-\frac{1}{8}(1~2)(3~4)_\Sigma-\frac{1}{8}(1~2~3~4)_\Sigma\\
           &=\frac{1}{8}\left(-D_{\fkt^{(4)}_1}-D_{\fkt^{(4)}_3}+D_{\fkt^{(4)}_4}+3D_{\fkt^{(4)}_5}\right).
\end{align*}
Hence 
\begin{align}
  \phi_4((3,1))&=\phi_4(2)e_{(3,1)}\nonumber\\
             &=e_{(3,1)}+\frac{1}{8\dd(\dd+2)}\left(\dd D_{\fkt^{(2)}_1}+2(\dd+2)D_{\fkt^{(2)}_2}-\dd D_{\fkt^{(2)}_3}-4D_{\fkt^{(2)}_4}-4(\dd+1)D_{\fkt^{(2)}_5}\right).\label{eq:phi_4(3,1)}
\end{align}
From \eqref{eq:phi_4(3,1)} above we see that when $\dd=0$ or $-2$, the idempotent $\phi_4((3,1))$ is no longer well-defined. The coefficients that blow up are attached to the diagrams with two propagating lines, signifying the appearance of $L_4(3,1)$ as a submodule of $\Delta_4(\mu)$ for $\mu\vdash2$.

Finally we will show that a sum of idempotents that individually are not defined at a certain value of $\dd$, can in fact be well defined for this $\dd$. In particular we will compute
\begin{equation}\label{eq:idempsum}
\phi_4(0)-\phi_4(2)+\phi_4((3,1))
\end{equation}
and show that it is well defined at $\delta=-2$, even though each constituent is not. Since each part is a linear combination of the $D_{\fkt^{(i)}_j}$ we can sum each of the corresponding coefficients. Using the order of the $\fkt^{(i)}_j$ from Section \ref{sec:exsplit} we have
\begin{align*}
  b_{\fkt^{(0)}_1}-c_{\fkt^{(0)}_1}&=\frac{4}{\dd(\dd-2)(\dd+4)}\\
  b_{\fkt^{(0)}_2}-c_{\fkt^{(0)}_2}&=\frac{-2(\dd+2)}{\dd(\dd-2)(\dd+4)}\\
  -c_{\fkt^{(2)}_1}+\frac{1}{8(\dd+2)}&=\frac{\dd}{8(\dd-2)(\dd+4)}\\
  -c_{\fkt^{(2)}_2}+\frac{1}{4\dd}&=\frac{\dd+2}{4(\dd-2)(\dd+4)}\\
  -c_{\fkt^{(2)}_3}-\frac{1}{8(\dd+2)}&=-\frac{\dd+8}{8(\dd-2)(\dd+4)}\\
  -c_{\fkt^{(2)}_4}-\frac{1}{2\dd(\dd+2)}&=-\frac{1}{2(\dd-2)(\dd+4)}\\
  -c_{\fkt^{(2)}_5}-\frac{\dd+1}{2\dd(\dd+2)}&=\frac{\dd+3}{2(\dd-2)(\dd+4)}.
\end{align*}
Note that the $D_{\fkt^{(4)}_j}$ appear only in $\phi_4((3,1))$, so we have omitted their coefficients here. We see then that \eqref{eq:idempsum} is well defined at $\dd=-2$. Since the element $\phi_4(0)$ kills all modules $\Delta_4(\lambda)$ with $|\lambda|=0$ and $\phi_4(2)$ kills all $\Delta_4(\lambda)$ with $|\lambda|\leq2$, this sum will kill all cell modules except $\Delta_4((3,1))$, $\Delta_4((2))$ and $\Delta_4((1^2))$. We have already seen that when $\dd=-2$ there is a homomorphism $\Delta_4((3,1))\rightarrow\Delta_4(\mu)$ for some $\mu\vdash2$, and from the block characterisation of \cite{cdm} we see that in fact $\mu=(1^2)$. From the same characterisation we see that $\Delta_4((2))$ is alone in its block. Therefore the sum \eqref{eq:idempsum} corresponds to a union of these two blocks of $B_4$.
\setcounter{section}{0}
\renewcommand{\thesection}{\Alph{section}}
\section{The case $n=6$}

\allowdisplaybreaks
The examples above illustrate the method of computing central
idempotents and how to glean information about representation theory
from them. In this supplementary section we calculate the splitting
idempotents in $B_6$, a $10395$-dimensional algebra, to show that the
method we derived does indeed give idempotents that were previously
inaccessible.

Starting with the splitting idempotent for $\overline{J_6}(0)$, we have
\[\fku^{(0)}_1=\young(NSNSNS),\hspace{1em}\fku^{(0)}_2=\young(NSNS,NS),\hspace{1em}\text{and }\fku^{(0)}_3=\young(NS,NS,NS).\]
Then $\phi_6(0)=1+\sum_{i=1}^3\alpha^{(0)}_iD_{\fku^{(0)}_i}$, where
\begin{align*}
  \alpha^{(0)}_1&=-\frac{2}{\dd(\dd-2)(\dd-1)(\dd+2)(\dd+4)},\\
  \alpha^{(0)}_2&=\frac{1}{\dd(\dd-2)(\dd-1)(\dd+4)},\\
  \alpha^{(0)}_3&=\frac{\dd^2+3\dd-2}{\dd(\dd-2)(\dd-1)(\dd+2)(\dd+4)}.
\end{align*}
Next, for the splitting idempotent for $\overline{J_6}(2)$ we have (in addition to the above)
\[\fku^{(2)}_1=\young(NSNSPP),\hspace{.7em}\fku^{(2)}_2=\young(NSNPSP),\hspace{.7em}\fku^{(2)}_3=\young(NSPNPS),\hspace{0.7em}\fku^{(2)}_4=\young(NSPNSP),\]
\[\fku^{(2)}_5=\young(NSNSP,P),\hspace{1em}\fku^{(2)}_6=\young(NSNS,PP),\hspace{1em}\fku^{(2)}_7=\young(NSPP,NS),\hspace{1em}\fku^{(2)}_8=\young(NPSP,NS),\]
\[\fku^{(2)}_9=\young(NSNS,P,P),\hspace{1em}\fku^{(2)}_{10}=\young(NSP,NSP),\hspace{1em}\fku^{(2)}_{11}=\young(NSP,NS,P),\hspace{1em}\fku^{(2)}_{12}=\young(NS,NS,PP),\hspace{1em}\text{and }\fku^{(2)}_{13}=\young(NS,NS,P,P).\]
Thus $\phi_6(2)=1+\sum_{i=1}^3\beta^{(0)}_iD_{\fku^{(0)}_i}+\sum_{i=1}^{13}\beta^{(2)}_iD_{\fku^{(2)}_i}$, where
\begin{align*}
  \beta^{(0)}_1&=\frac{13\dd^2+25\dd+18}{(\dd-3)(\dd-2)(\dd-1)(\dd+1)(\dd+2)(\dd+4)(\dd+6)},\\
  \beta^{(0)}_2&=-\frac{4(\dd^2+3\dd+3)}{(\dd-3)(\dd-2)(\dd-1)(\dd+1)(\dd+4)(\dd+6)},\\
  \beta^{(0)}_3&=\frac{2(\dd^4+7\dd^3+13\dd^2+13\dd-6)}{(\dd-3)(\dd-2)(\dd-1)(\dd+1)(\dd+2)(\dd+4)(\dd+6)},\\
  \\
  \beta^{(2)}_1&=\frac{3(\dd^2+\dd+2)}{\dd(\dd-3)(\dd-2)(\dd+1)(\dd+2)(\dd+4)(\dd+6)},\\
  \beta^{(2)}_2&=\frac{5\dd+6}{\dd(\dd-3)(\dd-2)(\dd+1)(\dd+4)(\dd+6)},\\
  \beta^{(2)}_3&=\frac{5\dd+6}{\dd(\dd-3)(\dd-2)(\dd+1)(\dd+4)(\dd+6)},\\
  \beta^{(2)}_4&=\frac{2(2\dd^2+3\dd-6)}{\dd(\dd-3)(\dd-2)(\dd+1)(\dd+2)(\dd+4)(\dd+6)},\\
  \beta^{(2)}_5&=-\frac{2\dd^3+10\dd^2+3\dd-6}{\dd(\dd-3)(\dd-2)(\dd+1)(\dd+2)(\dd+4)(\dd+6)},\\
  \beta^{(2)}_6&=-\frac{4(5\dd+6)}{\dd(\dd-3)(\dd-2)(\dd+1)(\dd+2)(\dd+4)(\dd+6)},\\
  \beta^{(2)}_7&=-\frac{(\dd+3)(\dd^2+\dd+2)}{\dd(\dd-3)(\dd-2)(\dd+1)(\dd+2)(\dd+4)(\dd+6)},\\
  \beta^{(2)}_8&=-\frac{2}{(\dd-3)(\dd-2)(\dd+1)(\dd+6)},\\
  \beta^{(2)}_9&=\frac{\dd^4+7\dd^3+8\dd^2-8\dd-24}{\dd(\dd-3)(\dd-2)(\dd+1)(\dd+2)(\dd+4)(\dd+6)},\\
  \beta^{(2)}_{10}&=-\frac{\dd^3+6\dd^2+18\dd+12}{\dd(\dd-3)(\dd-2)(\dd+1)(\dd+2)(\dd+4)(\dd+6)},\\
  \beta^{(2)}_{11}&=\frac{\dd^4+7\dd^3+7\dd^2-11\dd-6}{\dd(\dd-3)(\dd-2)(\dd+1)(\dd+2)(\dd+4)(\dd+6)},\\
  \beta^{(2)}_{12}&=\frac{8}{(\dd-3)(\dd-2)(\dd+1)(\dd+2)(\dd+6)},\\
  \beta^{(2)}_{13}&=-\frac{\dd^4+8\dd^3+7\dd^2-40\dd-44}{(\dd-3)(\dd-2)(\dd+1)(\dd+2)(\dd+4)(\dd+6)}.
\end{align*}
Finally, we compute the splitting idempotent for $\overline{J_6}(4)$. The remaining tableaux to consider are
\[\fku^{(4)}_1=\young(NSPPPP),\hspace{1em}\fku^{(4)}_2=\young(NPSPPP),\hspace{1em}\fku^{(4)}_3=\young(NPPSPP),\hspace{1em}\fku^{(4)}_4=\young(NSPPP,P),\]
\[\fku^{(4)}_5=\young(NPSPP,P),\hspace{1em}\fku^{(4)}_6=\young(NSPP,PP),\hspace{1em}\fku^{(4)}_7=\young(NPSP,PP),\hspace{1em}\fku^{(4)}_8=\young(PPPP,NS),\]
\[\fku^{(4)}_9=\young(NSPP,P,P),\hspace{1em}\fku^{(4)}_{10}=\young(NPSP,P,P),\hspace{1em}\fku^{(4)}_{11}=\young(NSP,PPP),\hspace{1em}\fku^{(4)}_{12}=\young(NSP,PP,P),\]
\[\fku^{(4)}_{13}=\young(PPP,NS,P),\hspace{1em}\fku^{(4)}_{14}=\young(NSP,P,P,P),\hspace{1em}\fku^{(4)}_{15}=\young(NS,PP,PP),\hspace{1em}\fku^{(4)}_{16}=\young(NS,PP,P,P),\hspace{1em}\text{and }\fku^{(4)}_{17}=\young(NS,P,P,P,P).\]
Therefore $\phi_6(4)=1+\sum_{i=1}^3\gamma^{(0)}_iD_{\fku^{(0)}_i}+\sum_{i=1}^{13}\gamma^{(2)}_iD_{\fku^{(2)}_i}+\sum_{i=1}^{17}\gamma^{(4)}_iD_{\fku^{(4)}_i}$, where
\begin{align*}
  \gamma^{(0)}_1&=-\frac{17\dd-18}{(\dd-4)(\dd-3)(\dd-2)(\dd+1)(\dd+6)(\dd+8)},\\
  \gamma^{(0)}_2&=\frac{3\dd^3+23\dd^2+66\dd-24}{(\dd-4)(\dd-3)(\dd-2)(\dd+1)(\dd+4)(\dd+6)(\dd+8)},\\
  \gamma^{(0)}_3&=-\frac{\dd^4+10\dd^3+19\dd^2-2\dd+408}{(\dd-4)(\dd-3)(\dd-2)(\dd+1)(\dd+4)(\dd+6)(\dd+8)},\\
  \\
  \gamma^{(2)}_1&=-\frac{7\dd^5+11\dd^4-27\dd^3-78\dd^2+216\dd-192}{\dd(\dd-4)(\dd-3)(\dd-2)(\dd-1)(\dd+1)(\dd+2)(\dd+4)(\dd+6)(\dd+8)},\\
  \gamma^{(2)}_2&=-\frac{11\dd^4+47\dd^3-6\dd^2-208\dd-96}{(\dd-4)(\dd-3)(\dd-2)(\dd-1)(\dd+1)(\dd+2)(\dd+4)(\dd+6)(\dd+8)},\\
  \gamma^{(2)}_3&=-\frac{11\dd^4+47\dd^3-6\dd^2-208\dd-96}{(\dd-4)(\dd-3)(\dd-2)(\dd-1)(\dd+1)(\dd+2)(\dd+4)(\dd+6)(\dd+8)},\\
  \gamma^{(2)}_4&=-\frac{2(4\dd^5+6\dd^4-29\dd^3+26\dd^2-136\dd+192)}{\dd(\dd-4)(\dd-3)(\dd-2)(\dd-1)(\dd+1)(\dd+2)(\dd+4)(\dd+6)(\dd+8)},\\
  \gamma^{(2)}_5&=\frac{5\dd^6+31\dd^5-62\dd^4-233\dd^3+298\dd^2+216\dd-192}{\dd(\dd-4)(\dd-3)(\dd-2)(\dd-1)(\dd+1)(\dd+2)(\dd+4)(\dd+6)(\dd+8)},\\
  \gamma^{(2)}_6&=\frac{4(15\dd^2+10\dd-88)}{(\dd-4)(\dd-3)(\dd-2)(\dd-1)(\dd+2)(\dd+4)(\dd+6)(\dd+8)},\\
  \gamma^{(2)}_7&=\frac{\dd^6+7\dd^5+31\dd^4-47\dd^3-78\dd^2+152\dd-192}{\dd(\dd-4)(\dd-3)(\dd-2)(\dd-1)(\dd+1)(\dd+2)(\dd+4)(\dd+6)(\dd+8)},\\
  \gamma^{(2)}_8&=\frac{2(\dd^5+9\dd^4+30\dd^3-52\dd-240)}{(\dd-4)(\dd-3)(\dd-2)(\dd-1)(\dd+1)(\dd+2)(\dd+4)(\dd+6)(\dd+8)},\\
  \gamma^{(2)}_9&=-\frac{3\dd^6+26\dd^5-29\dd^4-400\dd^3+192\dd^2+1256\dd-544}{(\dd-4)(\dd-3)(\dd-2)(\dd-1)(\dd+1)(\dd+2)(\dd+4)(\dd+6)(\dd+8)},\\
  \gamma^{(2)}_{10}&=\frac{\dd^6+9\dd^5+70\dd^4+78\dd^3-588\dd^2-80\dd+384}{\dd(\dd-4)(\dd-3)(\dd-2)(\dd-1)(\dd+1)(\dd+2)(\dd+4)(\dd+6)(\dd+8)},\\
  \gamma^{(2)}_{11}&=-\frac{\dd^7+10\dd^6+20\dd^5-42\dd^4-249\dd^3+42\dd^2+536\dd-192}{\dd(\dd-4)(\dd-3)(\dd-2)(\dd-1)(\dd+1)(\dd+2)(\dd+4)(\dd+6)(\dd+8)},\\
  \gamma^{(2)}_{12}&=-\frac{8(\dd^5+7\dd^4+36\dd^3-26\dd^2-240\dd+96)}{\dd(\dd-4)(\dd-3)(\dd-2)(\dd-1)(\dd+1)(\dd+2)(\dd+4)(\dd+6)(\dd+8)},\\
  \gamma^{(2)}_{13}&=\frac{\dd^8+11\dd^7+7\dd^6-171\dd^5-148\dd^4+716\dd^3+16\dd^2-192\dd+768}{\dd(\dd-4)(\dd-3)(\dd-2)(\dd-1)(\dd+1)(\dd+2)(\dd+4)(\dd+6)(\dd+8)},\\
  \\
  \gamma^{(4)}_1&=-\frac{\dd^3-14\dd^2-28\dd-48}{(\dd-4)(\dd-3)(\dd-2)(\dd+1)(\dd+2)(\dd+4)(\dd+6)(\dd+8)},\\
  \gamma^{(4)}_2&=-\frac{2(2\dd^4+4\dd^3-11\dd^2-18\dd-40)}{(\dd-4)(\dd-3)(\dd-2)(\dd-1)(\dd+1)(\dd+2)(\dd+4)(\dd+6)(\dd+8)},\\
  \gamma^{(4)}_3&=-\frac{2(3\dd^4+12\dd^3-16\dd^2-128\dd+192)}{\dd(\dd-4)(\dd-3)(\dd-2)(\dd-1)(\dd+2)(\dd+4)(\dd+6)(\dd+8)},\\
  \gamma^{(4)}_4&=\frac{\dd^4+4\dd^3-22\dd^2-32\dd+28}{(\dd-3)(\dd-2)(\dd-1)(\dd+1)(\dd+2)(\dd+4)(\dd+6)(\dd+8)},\\
  \gamma^{(4)}_5&=\frac{3\dd^5+20\dd^4-37\dd^3-200\dd^2+132\dd+208}{(\dd-4)(\dd-3)(\dd-2)(\dd-1)(\dd+1)(\dd+2)(\dd+4)(\dd+6)(\dd+8)},\\
  \gamma^{(4)}_6&=\frac{2(13\dd^3+10\dd^2-62\dd-24)}{(\dd-4)(\dd-3)(\dd-2)(\dd-1)(\dd+1)(\dd+2)(\dd+4)(\dd+6)(\dd+8)},\\
  \gamma^{(4)}_7&=\frac{4(10\dd^4+23\dd^3-120\dd^2-72\dd+96)}{\dd(\dd-4)(\dd-3)(\dd-2)(\dd-1)(\dd+1)(\dd+2)(\dd+4)(\dd+6)(\dd+8)},\\
  \gamma^{(4)}_8&=\frac{8(\dd^3-14\dd^2-28\dd-48)}{\dd(\dd-4)(\dd-3)(\dd-2)(\dd+1)(\dd+2)(\dd+4)(\dd+6)(\dd+8)},\\
  \gamma^{(4)}_9&=-\frac{\dd^4+4\dd^3-29\dd^2-2\dd+100}{(\dd-4)(\dd-3)(\dd-2)(\dd+1)(\dd+2)(\dd+6)(\dd+8)},\\
  \gamma^{(4)}_{10}&=-\frac{2(\dd^7+9\dd^6-9\dd^5-151\dd^4+20\dd^3+432\dd^2+16\dd-192)}{\dd(\dd-4)(\dd-3)(\dd-2)(\dd-1)(\dd+1)(\dd+2)(\dd+4)(\dd+6)(\dd+8)},\\
  \gamma^{(4)}_{11}&=\frac{21(\dd^2+2\dd-4)}{(\dd-3)(\dd-2)(\dd-1)(\dd+1)(\dd+2)(\dd+4)(\dd+6)(\dd+8)},\\
  \gamma^{(4)}_{12}&=-\frac{4(3\dd^3+17\dd^2-27\dd-76)}{(\dd-4)(\dd-3)(\dd-2)(\dd+1)(\dd+2)(\dd+4)(\dd+6)(\dd+8)},\\
  \gamma^{(4)}_{13}&=-\frac{6}{(\dd-2)(\dd-1)(\dd+2)(\dd+4)(\dd+8)},\\
  \gamma^{(4)}_{14}&=\frac{\dd^7+10\dd^6-8\dd^5-212\dd^4-11\dd^3+1042\dd^2+60\dd-1008}{(\dd-4)(\dd-3)(\dd-2)(\dd-1)(\dd+1)(\dd+2)(\dd+4)(\dd+6)(\dd+8)},\\
  \gamma^{(4)}_{15}&=-\frac{16(13\dd^3+10\dd^2-62\dd-24)}{\dd(\dd-4)(\dd-3)(\dd-2)(\dd-1)(\dd+1)(\dd+2)(\dd+4)(\dd+6)(\dd+8)},\\
  \gamma^{(4)}_{16}&=\frac{4(\dd^5+8\dd^4-\dd^3-50\dd^2-16\dd+96)}{\dd(\dd-4)(\dd-3)(\dd-2)(\dd+1)(\dd+2)(\dd+4)(\dd+6)(\dd+8)},\\
  \gamma^{(4)}_{17}&=-\frac{\dd^9+11\dd^8-7\dd^7-295\dd^6-106\dd^5+2252\dd^4+352\dd^3-4464\dd^2+96\dd+1152}{\dd(\dd-4)(\dd-3)(\dd-2)(\dd-1)(\dd+1)(\dd+2)(\dd+4)(\dd+6)(\dd+8)}.
\end{align*}

\section{Efficiency of the construction}\label{sec:leducram}

The expression of $B_n(\delta)$ as a multimatrix algebra in \cite{leducram} allows one to calculate the primitive central idempotents of $B_n$ directly by summing the elements corresponding to certain paths in the Bratteli diagram, see Figure \ref{fig:bratteli} below.
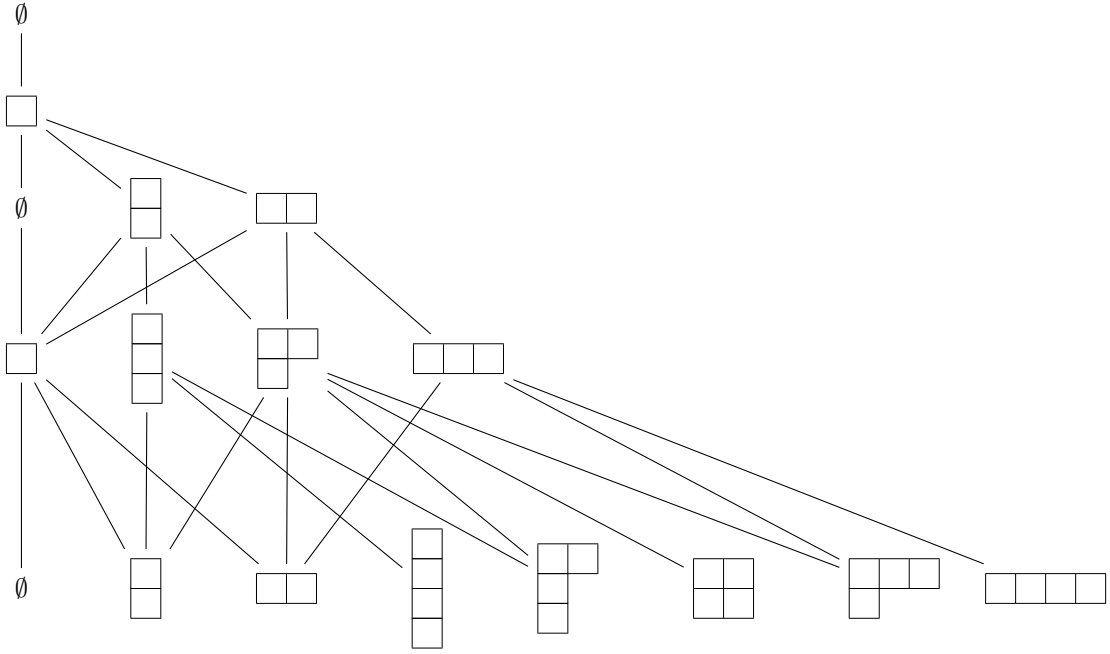
\begin{figure}[h!]\label{fig:bratteli}
\begin{tikzpicture}

  \node (A0) {\;$\emptyset$\;};

  \node[below=2em of A0] (B1) {$\yng(1)$};
  
  \node[below=2em of B1] (C0) {\;$\emptyset$\;};
  \node[right= of C0] (C11) {$\yng(1,1)$};
  \node[right= of C11] (C2) {$\yng(2)$};

  \node[below= 4em of C0] (D1) {$\yng(1)$};
  \node[right=of D1] (D111) {$\yng(1,1,1)$};
  \node[right=of D111] (D21) {$\yng(2,1)$};
  \node[right=of D21] (D3) {$\yng(3)$};

  \node[below=7em of D1] (E0) {\;$\emptyset$\;};
  \node[right=of E0] (E11) {$\yng(1,1)$};
  \node[right=of E11] (E2) {$\yng(2)$};
  \node[right=of E2] (E1111) {$\yng(1,1,1,1)$};
  \node[right=of E1111] (E211) {$\yng(2,1,1)$};
  \node[right=of E211] (E22) {$\yng(2,2)$};
  \node[right=of E22] (E31) {$\yng(3,1)$};
  \node[right=1em of E31] (E4) {$\yng(4)$};

  \draw (A0) to (B1);

  \draw (B1) to (C0);
  \draw (B1) to (C2);
  \draw (B1) to (C11);

  \draw (C0) to (D1);
  \draw (C2) to (D3);
  \draw (C2) to (D1);
  \draw (C2) to (D21);
  \draw (C11) to (D1);
  \draw (C11) to (D21);
  \draw (C11) to (D111);

  \draw (D1) to (E0);
  \draw (D1) to (E11);
  \draw (D1) to (E2);
  \draw (D111) to (E11);
  \draw (D111) to (E1111);
  \draw (D111) to (E211);
  \draw (D21) to (E2);
  \draw (D21) to (E11);
  \draw (D21) to (E211);
  \draw (D21) to (E22);
  \draw (D21) to (E31);
  \draw (D3) to (E2);
  \draw (D3) to (E31);
  \draw (D3) to (E4);
\end{tikzpicture}
\caption{The Bratteli diagram of $B_4$.}
\end{figure}
In particular we have a basis of $B_n$ given by $\{E_{ST}\}$ where $(S,T)$ are pairs of paths from row $0$ to the same point in row $n$ of the Bratteli diagram. Multiplication of these elements is given by the rule
\[E_{ST}E_{UV}=\delta_{TU}E_{SV}.\] 
 For $n=3$, we have the following:
\begin{align*}
  P_1&=\emptyset\rightarrow\yng(1)\rightarrow\yng(1,1)\rightarrow\yng(1,1,1)\\
  \\
  Q_1&=\emptyset\rightarrow\yng(1)\rightarrow\yng(1,1)\rightarrow\yng(2,1)\\
  Q_2&=\emptyset\rightarrow\yng(1)\rightarrow\yng(2)\rightarrow\yng(2,1)\\
  \\
  R_1&=\emptyset\rightarrow\yng(1)\rightarrow\yng(2)\rightarrow\yng(3)\\
  \\
  S_1&=\emptyset\rightarrow\yng(1)\rightarrow\yng(1,1)\rightarrow\yng(1)\\
  S_2&=\emptyset\rightarrow\yng(1)\rightarrow\yng(2)\rightarrow\yng(1)\\
  S_3&=\emptyset\rightarrow\yng(1)\rightarrow\emptyset\rightarrow\yng(1)
\end{align*}
Now by \cite[Theorem 6.22]{leducram} we can express the generators $u_i,\sigma_i$ of $B_n$ as linear combinations of the $E_{ST}$ over the field $\overline{\Q(\dd)}$. In particular for $B_3$ we have
\begin{align*}
  u_1&=\dd E_{S_1S_1}\\
  s_1&=-E_{P_1P_1}+E_{R_1R_1}-E_{Q_1Q_1}+E_{Q_2Q_2}+E_{S_1S_1}-E_{S_2S_2}+E_{S_3S_3}\\
  u_2&=\frac{1}{\dd}E_{S_1S_1}+\frac{\dd-1}{2}E_{S_2S_2}+\frac{(x-1)(x+2)}{2x}E_{S_3S_3}+\frac{\sqrt{x(x-1)}}{\sqrt{2}x}(E_{S_1S_2}+E_{S_2S_1})\\&\pushright{+\frac{\sqrt{(x-1)(x+2)}}{\sqrt{2}x}(E_{S_1S_3}+E_{S_3S_1})+\frac{\sqrt{x(x-1)^2(x+2)}}{2x}(E_{S_2S_3}+E_{S_3S_2})}\\
  s_2&=-E_{P_1P_1}+E_{R_1R_1}+\frac{1}{2}(E_{Q_1Q_1}-E_{Q_2Q_2})+\frac{\sqrt{3}}{2}(E_{Q_1Q_2}+E_{Q_2Q_1})+\frac{1}{\dd}E_{S_1S_1}+\frac{1}{2}E_{S_2S_2}\\&\pushright{+\frac{(x-2)}{2x}E_{S_3S_3}-\frac{\sqrt{x(x-1)}}{\sqrt{2}x}(E_{S_1S_2}+E_{S_2S_1})+\frac{\sqrt{(x-1)(x+2)}}{\sqrt{2}x}(E_{S_1S_3}+E_{S_3S_1})}\\&\pushright{+\frac{\sqrt{x(x-1)^2(x+2)}}{2x(x-1)}(E_{S_2S_3}+E_{S_3S_2})}.
\end{align*}
Since the set $\{u_1,s_1,u_2,s_2\}$ generates a $\Z[\dd]$-basis for $B_3$ we can find this basis in terms of the $E_{ST}$, and hence find an expression for the $E_{ST}$ in terms of the standard diagram basis. To calculate the primitive central idempotent $\phi_n(\lambda)$ with this basis we must sum the elements $E_{SS}$ where $S$ is a path ending at $\lambda$. For instance for $\phi_3((1))$ we find the sum
\begin{multline*}
  E_{S_1S_1}+E_{S_2S_2}+E_{S_3S_3}=\frac{\dd+1}{(\dd-1)(\dd+2)}(u_1+u_2+s_1u_2s_1)\\-\frac{1}{(\dd-1)(\dd+2)}(u_1u_2+u_2u_1+u_1s_2+u_2s_1+s_1u_2+s_2u_1).
\end{multline*}
This is simply the element $X_4(0)$ in the notation of this paper, which is already a much easier calculation. Moreover in order to write the generators of the algebra in terms of the $E_{ST}$ we must calculate a coefficient for each pair of paths $(S,T)$ ending at the same partition. The number of such pairs grows dramatically with $n$, as does the dimension of the algebra $B_n$ and hence the calculation to convert from one basis to the other. Finally the coefficients in the intermediate steps do not reside in some integral or otherwise ``nice'' ring, which is a property of the method described in this paper.

\medskip\medskip
\noindent {\bf Acknowledgements.}
We would like to thank Rosa Orellana for a helpful comment on the
draft and  EPSRC for financial support under
grant EP/L001152/1.

\bibliographystyle{plain}
\bibliography{KMP-BrIde}
\end{document}